\PassOptionsToPackage{compatibility=false}{caption}
\documentclass[letterpaper, 10 pt, conference]{ieeeconf}

\IEEEoverridecommandlockouts
\usepackage[utf8]{inputenc}
\usepackage[american]{babel}
\usepackage{csquotes}
\usepackage[letterpaper,left=48pt,bottom=43pt,right=48pt,top=57pt]{geometry}

\usepackage{comment}
\usepackage{url}
\usepackage{graphicx}
\usepackage{subcaption}
\usepackage{multicol}
\usepackage{xspace}

\usepackage{amsmath, amssymb, amsfonts, amsthm}
\usepackage{mathtools}
\usepackage{multirow}
\usepackage{tabularx}
\usepackage[binary-units]{siunitx}

\newcolumntype{C}{>{\centering\arraybackslash}X}
\newcolumntype{L}{>{\raggedright\arraybackslash}X}
\usepackage{siunitx}
\usepackage{listings}
\usepackage{float}

\newcommand{\R}[1][\empty]{\mathbb{R}^{#1}}

\newtheorem{definition}{Definition}
\newtheorem{lemma}{Lemma}
\newtheorem{theorem}{Theorem}
\newtheorem{example}{Example}

\newtheorem{proposition}{Proposition}
\newtheorem{remarkth}[definition]{Remark}
\newenvironment{remark}{\begin{remarkth}\upshape}{\end{remarkth}}

\usepackage{tikz}
\usetikzlibrary{matrix,arrows,calc,positioning,shapes,decorations.pathreplacing, angles, quotes, patterns}
\tikzstyle{vertex}=[circle,fill=black!20,minimum size=15pt,inner sep=0pt]
\tikzstyle{selected vertex} = [vertex, fill=red!24]
\tikzstyle{edge} = [draw,thick,-]
\tikzstyle{dedge} = [draw,thick,<->]
\tikzstyle{shadowdedge} = [draw, dotted,->]
\tikzstyle{weight} = [font=\small]
\tikzstyle{selected edge} = [draw,line width=5pt,-,red!50]
\tikzstyle{ignored edge} = [draw,line width=5pt,-,black!20]

\usepackage{booktabs}

\title{\LARGE \bf Contracting Forced Lagrangian and Contact Lagrangian Systems: application to nonholonomic systems with symmetries}
\author{Alexandre Anahory Simoes$^{1}$\thanks{$^{1}$A. Anahory Simoes is with the School of Science and Technology, IE University, Spain. email:  {\tt\small alexandre.anahory@ie.edu}}%
, and Leonardo Colombo$^{2}$%
\thanks{$^{2}$Leonardo Colombo is Centre for Automation and Robotics (CSIC-UPM), Centre for Automation and Robotics (CSIC-UPM), Ctra. M300 Campo Real, Km 0,200, Arganda
del Rey - 28500 Madrid, Spain. email:  {\tt\small leo.colombo@icmat.es}}%
\thanks{The authors acknowledge financial support from Grant PID2019-106715GB-C21 funded by MCIN/AEI/ 10.13039/501100011033.}
}
\begin{document}
\newgeometry{left=54pt,bottom=54pt,right=54pt,top=54pt}

\maketitle
\thispagestyle{empty}
\pagestyle{empty}

\begin{abstract}
    In this paper we address the problem of identifying contracting systems among dynamical systems appearing in mechanics. First, we introduce a sufficient condition to identify contracting systems in a general Riemannian manifold. Then, we apply this technique to establish that a particular type of dissipative forced mechanical system is contracting, while stating immediate consequences of this fact for the stability of these systems. Finally, we use the previous results to study the stability of particular types of Contact and Nonholonomic systems.
\end{abstract}


\section{Introduction}
Contraction theory has appeared several times in the control literature (see \cite{Lohmiller, Slotine, Sontag2010, Aghannan} and \cite{Jouffroy} for a survey on the subject). This subject has deserved some attention due to its potential applications in many areas ranging from bio-molecular systems to the design of controllers for mechanical systems (e.g., \cite{Pavlov,Russo,Sanfelice,Lohmiller2}). In effect, contraction theory is a useful tool in nonlinear analysis to demonstrate stability properties of dynamical systems. Dynamical systems exhibiting contracting behaviour collapse exponentially fast into a single equilibrium point due to the fact that in a neighbourhood of an equilibrium point two initially distant solutions converge at exponential rate. As far as we know, the first time that contraction theory was studied in a coordinate-free setting using Riemannian manifolds was in \cite{Simpson}, where the authors state and prove all the results using only intrinsic properties that do not rely on the particular choice of coordinates.


This paper tries to address the following fundamental question: ``On which circumstances does a mechanical system exhibit contracting behaviour?". Some examples have been examined in the literature, for instance, in \cite{Simpson}, the authors study the case of the damped harmonic oscillator and, in a slightly different setting, \cite{Lohmiller2} examines specific examples of controlled mechanical systems. However, this question, as it is posed, is too general. On one hand, the important class of conservative mechanical systems do not have contracting behaviour. On the other hand, the examples of non-conservative mechanical systems are still too wide. In this paper, we restrict ourselves to a particular type of forced mechanical systems and we prove that, under some mild assumptions, it exhibits contracting behaviour. Then, we examine the implications of this result in the contact formulation of dissipative system and in a special case of nonholonomic systems with symmetries.

Forced mechanical systems appear frequently in mechanical problems where a force that cannot be absorbed by the potential function is acting on the system. Typical examples of forces include friction, air resistance or external factors such as water currents or wind disturbances (see \cite{Abraham1978, bullo, colombo2021forced} for instance). We are specially interested in the former case. Both friction as well as air resistance are examples of dissipative forced systems where the force causes a consistent loss of total energy. The equations of motion for these systems \textit{a priori} are non-variational, i.e., they are not derivable from minimizing the action function associated with a Lagrangian. 


Contact Lagrangian systems have been developed in the last few years as an attempt to deal with dissipative systems in a unified framework (see \cite{Bravetti2017, deLeon2020, Gaset2019, Gaset2020, Kaufman, Ciaglia18, colombo2022contact}). In this framework, the action of the mechanical system is added as a variable to the phase space. Through a modification of the principle of least work (\cite{deLeon2020, Herglotz1930, lopez2022nonsmooth}), the modified variational principle outputs trajectories of a dissipative system evolving in a contact manifold \cite{marle}. This framework sheds some light on qualitative features possessed by dissipative systems. One of its main advantages is the existence of symmetries \cite{deLeon2020} and the possibility to apply reduction by symmetries techniques. It is worth mentioning that contact systems have also been used to model the evolution of thermodynamic variables \cite{ Bravetti2018,  Hernandez1998}, quantum mechanics \cite{Ciaglia18} and the evolution of finite state dissipative quantum systems \cite{Bravettiquantum}.


Some mechanical systems have a restriction either on the configurations the system may assume or in the velocities the system is allowed to go. Systems with such restrictions are generally called constrained systems \cite{LR2011}. In particular, nonholonomic  systems \cite{Neimark, Bloch, LMdD1996} are systems for which the velocity is restricted. Similarly to a forced system, nonholonomic equations are non-variational. Some special examples of nonholonomic systems possess symmetries that allow to reduce the number of equations of motion (\cite{BKMM96, BMZ2005, koiller92}). One of the most studied classes of examples is that of Chaplygin systems, in which both the Lagrangian function as well as the constraints are invariant with respect to some symmetry. Then one can take advantage of the symmetries to reduce the number of the equations of motion and the resulting (reduced) mechanical system happens to be a forced mechanical system. We take advantage of this case to apply contraction theory to Chaplygin nonholonomic systems. See \cite{cantrijn,CaLeMaMa, BaSn, SCSarlet99} for a description of Chaplygin systems.


In this paper, we advance one step further in the direction of the coordinate-free description of contracting systems \cite{Simpson} and prove a sufficient condition for a system to be contracting on a general Riemannian manifold. Then, we show that forced mechanical systems satisfying some assumptions are contracting with respect to a Riemannian metric called the \textit{mechanical contraction metric}. Hopefully, this result is the first building block of a general theory characterizing the situation for all forced mechanical systems. Then we discuss the implications concerning the stability of these systems and translate the result to the case of Contact Lagrangian systems. It is worth noting that we discuss briefly the implications for the stabilization of a Contact Lagrangian system, a question that remains to explore in the Contact literature. Finally, we take advantage of the existing symmetries in Chaplygin systems to identify examples that exhibit contracting behavior and examine the implications on the complete set of nonholonomic equations of motion. In particular, we address the question of stability for these systems.

The remainder of the paper is organized as follows: in Section 2 we review the mathematical machinery used throughout the text. In Section 3, we review contraction systems in general Riemannian manifolds and we prove a new sufficient condition to find a contraction metric in Theorem \ref{main:theorem}. In Section 4, we get down to a particular type of forced Lagrangian system and, as the main result of the paper, we prove in Theorem \ref{mechanical:theorem} that under some assumptions these systems exhibit contracting behavior. In Section 5, we address a similar problem using the framework of Contact Lagrangian systems. And finally, in Section 6, we apply our main result to a special type of nonholonomic system with symmetries: Chaplygin systems.

\section{Preliminaries on differential geometry and mechanical systems}

Suppose $Q$ is a manifold of dimension $n$. Throughout the text, $q^{i}$ will denote a particular choice of local coordinates on this manifold and $TQ$ denotes its tangent bundle, with $T_{q}Q$ denoting the tangent space at a specific point $q\in Q$ generated by the coordinate vectors $\frac{\partial}{\partial q^{i}}$. Usually $v_{q}$ denotes a vector at $T_{q}Q$ and, in addition, the coordinate chart $q^{i}$ induces a natural coordinate chart on $TQ$ denoted by $(q^{i},\dot{q}^{i})$. There is a canonical projection $\tau_{Q}:TQ \rightarrow Q$, sending each vector $v_{q}$ to the corresponding base point $q$. Note that in coordinates $\tau_{Q}(q^{i},\dot{q}^{i})=q^{i}$.

A vector field $X$ on $Q$ is a map assigning to each point $q\in Q$ a vector tangent to $q$, that is, $X(q)\in T_{q}Q$. In the context of mechanical systems, we find a special type of vector fields that are always defined on the tangent bundle $TQ$, considered as a manifold itself. A second-order vector field (SODE) $\Gamma$ on the tangent bundle $TQ$ is a vector field on the tangent bundle satisfying the property that $T\tau_{Q}\left(\Gamma (v_{q})\right) = v_{q}$. The expression of any SODE in coordinates is the following:
$$\Gamma(q^{i},\dot{q}^{i})= \dot{q}^{i}\frac{\partial}{\partial q^{i}} + f^{i}(q^{i},\dot{q}^{i}) \frac{\partial}{\partial \dot{q}^{i}},$$
where $f^{i}:TQ \rightarrow \mathbb{R}$ are $n$ smooth functions.

From now on, we will introduce concepts and properties that are valid for any manifold regardless of being a tangent bundle or not. To emphasize this, we will consider from now on any manifold $M$ of dimension $n$. An equilibrium point of a vector field $X$ on $M$ is a point $\bar{q} \in M$ for which $X(\bar{q})=0$. An integral curve of the vector field $X$ is a curve $\gamma:I\rightarrow M$ with $I\subseteq \mathbb{R}$ such that $\gamma'(t)=X(\gamma(t))$. The flow of $X$ is a map $\phi^{X}:I\times M \rightarrow M$ for which the map $t \mapsto \phi_{t}^{X}(q)$ is an integral curve of $X$ and $\phi_{0}^{X}(q)=q$.

A Riemannian metric $g$ on a manifold $M$ is a $(0,2)$ tensor such that for each fixed $q\in M$, the map $\mathbb{G}_{q}:T_{q}M \times T_{q}M \rightarrow \mathbb{R}$ is a symmetric positive-definite bilinear form on the vector space $T_{q}M$. The Riemannian metric induces a norm at each base point $q\in M$ which we will denote by $\|\cdot\|_{g}$. With the definition of a norm, one can measure the length of a curve $\gamma:[0,1]\rightarrow M$ on $M$: $L(\gamma) = \int_{0}^{1} \|\gamma'(t)\|_{g} \ dt.$
In addition, the distance between two points $q_{0}$ and $q_{1}$ is given by the infimum length among all piece-wise smooth curves connecting $q_{0}$ and $q_{1}$.

The Lie derivative of a function $f:M\rightarrow \mathbb{R}$ in the direction of the vector field $X$ on $M$ is given by the expression $\mathcal{L}_{X} f = df (X)$. One can extend this notion to differentiate vector fields and tensors. In the first case, the Lie derivative of a vector field $Y$ in the direction of $X$ is given by the Lie bracket $\mathcal{L}_{X} Y = [X,Y]$. And in the second case, the Lie derivative of a tensor $g$ in the direction of $X$ is the $(0,2)$ tensor for which
$$(\mathcal{L}_{X} g) (Y,Z) = X(g (Y,Z)) - g(\mathcal{L}_{X} Y, Z) - g(Y, \mathcal{L}_{X} Z).$$

The complete lift of a vector field $X$ on $M$ (\cite{LR2011,Yano1973}) is a vector field on the tangent bundle $TM$ denoted by $X^{c}$ whose flow is $T\phi_{t}^{X}$, where $\phi_{t}^{X}$ is the flow of $X$. If the local coordinates of $X$ are $X=X^{i}\frac{\partial}{\partial q^{i}}$, then
$$X^{c}=X^{i}\frac{\partial}{\partial q^{i}} + \dot{q}^{j}\frac{\partial X^{i}}{\partial q^{j}}\frac{\partial}{\partial \dot{q}^{i}}.$$

The possible configurations of a mechanical systems are usually described by a manifold $Q$ called the configuration manifold (for an introduction see \cite{Abraham1978,bullo}). Given a Riemannian metric $g$ in the configuration manifold $Q$, all the information about a mechanical system is contained in a function on the tangent bundle $TQ$ called the Lagrangian of the system. Given a Lagrangian $L:TQ \rightarrow \mathbb{R}$, the equations of motion are given by Euler-Lagrange equations
$$\frac{d}{dt}\frac{\partial L}{\partial \dot{q}^{i}} - \frac{\partial L}{\partial q^{i}}= 0.$$

If the mechanical system is subjected to external forces and impulses, this terms are dealt with independently from the Lagrangian. A force is a map $F:TQ \rightarrow T^{*}Q$ given in local coordinates by the expression $F(v_{q}) = F^{i}(v_{q})dq^{i}$. In the presence of external forces, the equations of motion must be modified according to
$$\frac{d}{dt}\frac{\partial L}{\partial \dot{q}^{i}} - \frac{\partial L}{\partial q^{i}}= F^{i}.$$

The kinetic energy of a system in a Riemannian manifold is given by a special Lagrangian function $L_{g}:TQ\rightarrow \mathbb{R}$ defined by $L_{g}(v_{q}) = \frac{1}{2}g(v_{q},v_{q}).$
A mechanical system is said to be of mechanical type or is said to have a mechanical Lagrangian if it is described by a Lagrangian of the type $L=L_{g} - V\circ \tau_{Q}$, where $V:Q \rightarrow \mathbb{R}$ is a smooth function called the potential energy. 

We will consider throughout the text, mechanical Lagrangians of the form
$$L(q,\dot{q}) = \frac{1}{2}\dot{q}^{T}M\dot{q} - \frac{1}{2}q^{T}Kq,$$
where $q, \dot{q} \in \mathbb{R}^{n}$. In consequence, our discussion will be valid either for systems with configuration manifold $Q = \mathbb{R}^{n}$ or, if not, within a coordinate chart of a more general $n$-dimensional manifold. In addition, we will consider an external dissipative force of the form $F(q,\dot{q}) = -b\dot{q}^{i} dq^{i},$ where $b>0$. In this way, the mechanical system satisfies the following equations of motion
$$\ddot{q} = M^{-1} \left(- Kq - b \dot{q}\right).$$
Note that the previous equations of motion can be seen as the integral curves of a SODE vector field $\Gamma$ (\cite{Abraham1978,bullo,LR2011,Bloch}) called a forced mechanical vector field.
This force map is called dissipative since the derivative of the energy of the mechanical system along the trajectories decreases. In general, the energy is given by the expression $E_{L} = \dot{q}^{i}\frac{\partial L}{ \partial \dot{q}^{i}} - L$. In this case, we obtain $\frac{d}{dt} E_{L} = -b\|\dot{q}\|^{2}<0.$

A contact Lagrangian system is able to encapsulate dissipative behaviour in a single Lagrangian function, at the price of working in a higher dimension. The dynamics of dissipative systems may be described by a function called contact Lagrangian $L:TQ\times \mathbb{R}\rightarrow \mathbb{R}$ and by the equations of motion
$$\frac{d}{dt}\frac{\partial L}{\partial \dot{q}^{i}} - \frac{\partial L}{\partial q^{i}}= \frac{\partial L}{\partial \dot{q}^{i}} \frac{\partial L}{\partial z}, \quad \dot{z}=L,$$
where $(q,\dot{q},z)$ are the coordinates on $TQ\times \mathbb{R}$ (see, e.g. \cite{deLeon2020, Herglotz1930}). Contact mechanical Lagrangians are of the type $L=L_{g} - V\circ \tau_{Q} - b z$ and, in particular, we will be interested in contact mechanical Lagrangians of the form
$$L(q,\dot{q},z) = \frac{1}{2}\dot{q}^{T}M\dot{q} - \frac{1}{2}q^{T}Kq - \frac{b}{m}z,$$
with $m, b>0$. The equations of motion are in this case
$$\ddot{q} = M^{-1} \left(- Kq - \frac{b}{m} M\dot{q}\right),$$
together with $\dot{z}=L$. This approach has some geometric advantages (\cite{Bravetti2017, deLeon2020, Gaset2019, Gaset2020, Kaufman, Ciaglia18,Lainz}).

Later, we will apply our main result to a special type of nonholonomic systems with symmetries: Chaplygin systems (see \cite{cantrijn,CaLeMaMa, BaSn, SCSarlet99}. A nonholonomic system is a mechanical system whose velocities are constrained. In this text, we will consider only linear constraints on the velocities, implying that the velocity of the system is always contained in a subspace of the tangent space. This subspace, called the constraint subspace, is usually denoted by $\mathcal{D}_{q}$ and is locally given by an expression of the type $\mu^{a}_{i}\dot{q}^{i}=0$. The nonholonomic equations of motion, as deduced from a Lagrangian function, are
$$\frac{d}{dt}\frac{\partial L}{\partial \dot{q}^{i}} - \frac{\partial L}{\partial q^{i}}= \lambda_{a}\mu^{a}_{i},$$
where $\lambda_{a}$ is a Lagrange multiplier that might be computed using the constraints (see \cite{Neimark, Bloch, LMdD1996}).

Now, suppose that $\Phi:G\times Q\rightarrow Q$ is an action of a Lie group $G$ on a smooth manifold $Q$, denoted by $\Phi(g,q)=\Phi_g (q)$. Under mild assumptions ($\Phi$ is a free and proper action) the natural projection $\pi:Q\rightarrow Q/G$ is a surjective submersion onto the reduced space $\overline{Q}:= Q/G$.

A generalized Chaplygin system is a nonholonomic system whose Lagrangian $L:TQ\rightarrow\R$ is a $G$-invariant regular function for the lifted action of $G$ on $TQ$, i.e., $L(T\Phi_{g}(v))=L(v), \ \forall \ v\in T_{q}Q$. Moreover, the constraints $\mathcal{D}_{q}$ are also $G-$ invariant, i.e., $T\Phi_{g}(\mathcal{D}_{q})=\mathcal{D}_{\Phi_{g}(q)}$ and satisfy $T_{q}Q=\mathcal{D}_{q} \oplus V_{q}$, where $V_{q}\subseteq T_{q}Q$ is the vertical subspace defined by the vectors $v_{q}$ whose projection by $T_{q}\pi$ vanishes. In this case, the constraint subspace is also called the horizontal subspace.

Given the properties satisfied by Chaplygin systems, every vector $X\in T_{q}Q$ can be uniquely written as $X=\text{hor}(X)+\text{ver}(X),$ where $\text{hor}(X)$ is an horizontal vector field (resp. $\text{ver}(X)$ is a vertical vector field), i.e., $\text{hor}(X)(q)\in \mathcal{D}_{q}$ (resp., $\text{ver}(X)(q)\in V_{q}$).

The horizontal lift of a vector field $Y$ on the reduced space $\overline{Q}$ is the unique horizontal vector field $Y^{H}$ on $Q$ that projects onto $Y$, i.e., $T_{q}\pi(Y^{H}(q))=Y(\pi(q))$. Analogously, the horizontal lift of a curve $\gamma$ on $\overline{Q}$ is the unique curve $\gamma^{H}$ on $Q$ that projects to $\gamma$ and whose tangent vector is horizontal at every instant.

Now, given a generalized Chaplygin system, the nonholonomic equation of motion can be seen as the integral curves of a vector field $\Gamma_{nh}\in \mathfrak{X}(\mathcal{D})$. This vector field reduces to a vector field $\overline{\Gamma}=T(T\pi)(\Gamma_{nh})$ called the reduced dynamics and whose integral curves are called the reduced trajectories. Conversely, the nonholonomic vector field is also the horizontal lift on the reduced vector field. In addition, the reduced vector field is a forced mechanical vector field for the reduced Lagrangian function $l:T\overline{Q}\rightarrow \mathbb{R}$ defined by $l(X(\pi(q)))=L(X^{H}(q))$ for any $X\in \mathfrak{X}(\overline{Q})$ and some force map $F:T\overline{Q} \rightarrow T^{*}\overline{Q}$ related with the constraints (see more details in \cite{cantrijn, CaLeMaMa}).

\section{Contracting systems}

First, we review the definition of a contracting system according to \cite{Simpson}. Note that, a contracting system is not only composed by a particular dynamics but also by an associated Riemannian metric and open subset.

\begin{definition}
	A contracting system is a tuple $(M,X,g,\mathcal{U})$ where $M$ is a smooth manifold, $X$ is a smooth vector field on $M$, $g$ is a Riemannian metric on $M$, called the contraction metric, and $\mathcal{U}$ is an open subset on $M$ such that there exists $\lambda>0$, called the \textit{contraction rate}, for which
	\begin{equation}\label{contraction:condition}
		\mathcal{L}_X g (v_{x},v_{x}) \leqslant -2\lambda g(v_{x},v_{x}), \quad \forall x\in \mathcal{U}, v_{x} \in T_{x}M.
	\end{equation}
\end{definition}

Under the hypothesis that the region $\mathcal{U}$ is $K$-reachable with $K\geqslant 1$, that is, for every pair of points $x_{0}$ and $x_{1}$ in $\mathcal{U}$ there exists a continuously differentiable curve $\gamma:[0,1]\rightarrow \mathcal{U}$ satisfying $\gamma(0)=x_{0}$, $\gamma(1)=x_{1}$ and $\ell_{g}(\gamma)\leqslant K d_{g}(x_{0},x_{1})$, we have several results describing the qualitative behaviour of a contracting system. The following is the main theorem in \cite{Simpson} describing the implications of a system being contracting.

\begin{theorem}
	Suppose that $(M,X,g, \mathcal{U})$ is a contracting system with contraction rate $\lambda>0$ and for some $K\geqslant 1$, $\mathcal{U}$ is a $K$-reachable forward $X$-invariant set. Furthermore, suppose that $X$ is forward complete on $\mathcal{U}$ and denote its flow by $\Phi_{t}(x)$. Then for each $x_{0}$, $x_{1} \in \mathcal{U}$ and $t\geqslant 0$ we have that
	\begin{equation*}
		d_{g}(\Phi_{t}(x_{0}),\Phi_{t}(x_{1})) \leqslant K e^{-\lambda t} d_{g}(x_{0},x_{1}).
	\end{equation*}
\end{theorem}

In particular, we are interested in some consequences of the previous theorem such as the following Proposition from \cite{Simpson}.

\begin{proposition}\label{contracting:lyapunov:prop}
	In the conditions of the previous theorem, if assume in addition that $(\mathcal{U},d_{g})$ is a complete metric space, then $X$ has a unique fixed point $\bar{x}\in \mathcal{U}$ and for each $x\in \mathcal{U}$, $\Phi_{t}(x) \rightarrow \bar{x}$ exponentially fast as $t\rightarrow +\infty$. Moreover,
	$$V(x)=\|X(x)\|_{g}^{2}$$
	is a strict local Lyapunov function for the unique fixed point.
\end{proposition}

So far, we have established the most important consequences implied by contracting behaviour. The remaining of the section will be dedicated to establish criteria to find contracting systems. Below, we prove a sufficient condition to identify contracting systems.

\begin{theorem}\label{main:theorem}
	If the symmetric $(0,2)$-tensor $\mathcal{L}_{X}g$ is negative definite on a compact set $\tilde{\mathcal{U}}$ then $(M,X,g,\mathcal{U})$ is a contracting system with $\mathcal{U}\subseteq \tilde{\mathcal{U}}$ an open subset.
\end{theorem}

\begin{proof}
	Suppose that $\mathcal{L}_{X}g$ is negative definite on $\tilde{\mathcal{U}}$. The proof is based on a topological argument. We will prove that if $\mathcal{L}_{X}g$ is negative definite then there exists a small enough $\lambda>0$ making the symmetric $(0,2)$-tensor $\mathcal{L}_{X}g + 2\lambda g$ negative definite.
	
	
	Indeed, suppose that $\{h_{i}\}_{i \in \mathbb{N}}$ is a sequence of symmetric $(0,2)$-tensors converging to a negative definite symmetric $(0,2)$-tensor $h$, with respect to the norm
	$$\|h\|= \sup_{x \in \tilde{\mathcal{U}}} \sup_{\|v_{x}\|_{g}=1} |h(v_{x},v_{x})|.$$
    Note that this norm is well-defined, since $\tilde{\mathcal{U}}$ is compact (see Lemma \ref{norm:proof} in the Appendix for a proof). The fact that $h_{i}\rightarrow h$ with respect to this norm means that for all $\varepsilon>0$ there exists $N\in \mathbb{N}$ such that $\|h_{i}-h\|<\varepsilon.$
	Therefore for any $v_{x}\in T_{x} M$ with $x \in \tilde{\mathcal{U}}$ and $\|v_{x}\|_{g}=1$ we have that
	$$|h_{i}(v_{x},v_{x}) - h(v_{x},v_{x})| < \|h_{i}-h\|<\varepsilon$$
	Choosing $\varepsilon>0$ such that $h(v_{x},v_{x})+\varepsilon < 0$ which is possible since the function $v_{x} \mapsto h (v_{x},x_{x})$ has a maximum on the compact set $\|v_{x}\|_{g}=1$, we guarantee that for $i>N$ we have that $h_{i}(v_{x},v_{x})<0$. Now for any $v_{x}\neq 0$ we have that
	$$h_{i}(v_{x},v_{x})=\|v_{x}\|_{g}^{2}h_{i}\left(\frac{v_{x}}{\|v_{x}\|_{g}},\frac{v_{x}}{\|v_{x}\|_{g}}\right)<0$$
	for $i>N$. Thus we have proved that the set of negative definite quadratic forms is open. Consider the curve
	$$\lambda \in [0,+\infty) \mapsto \mathcal{L}_{X}g + 2\lambda g.$$
	Then, since the set of negative definite quadratic forms is open, the curve must remain there for sufficiently small values of $\lambda$. Choosing any open subset $\mathcal{U}$ inside $\tilde{\mathcal{U}}$, the result follows.
	
\end{proof}

\begin{remark}
	The previous result gives us a simple way to determine whether a system is contracting. However, note that it gives us no estimate on which values the contraction rate may assume.
\end{remark}

A useful property to rewrite the contracting condition \eqref{contraction:condition} is the following.

\begin{lemma}\label{useful:lemma}
Given a vector field $X$ on a Riemannian manifold $M$ with metric $g$ we have that
	\begin{equation*}
		\mathcal{L}_X g (v_{x},v_{x}) = 2 X^{c}(L_{g}).
	\end{equation*}
\end{lemma}

\begin{proof}
    Let $(q^{i})$ be coordinates on $Q$ and $(q^{i},\dot{q}^{i})$ be the natural coordinates on $TQ$. Let $X=X^{i}\frac{\partial}{\partial q^{i}}$ and $L_{g}=\frac{1}{2}g_{ij}\dot{q}^{i}\dot{q}^{j}$. On one hand we have that
    \begin{equation*}
		X^{c}(L_{g})=\frac{1}{2}X^{k}\frac{\partial g_{ij}}{\partial q^{k}}\dot{q}^{i}\dot{q}^{j}+\dot{q}^{i}\frac{\partial X^{k}}{\partial q^{i}}g_{kj}\dot{q}^{j}.
    \end{equation*}
    On the other hand
    \begin{equation*}
		\mathcal{L}_{X}g=\left( X^{k}\frac{\partial g_{ij}}{\partial q^{k}}+g_{kj}\frac{\partial X^{k}}{\partial q^{i}}+g_{ik}\frac{\partial X^{k}}{\partial q^{j}} \right)\dot{q}^{i}\dot{q}^{j}.
	\end{equation*}
    Since $g$ is symmetric, we may interchange indices $i\leftrightarrow j$ in the last term to get
    \begin{equation*}
		\mathcal{L}_{X}g= X^{k}\frac{\partial g_{ij}}{\partial q^{k}}\dot{q}^{i}\dot{q}^{j}+g_{kj}\frac{\partial X^{k}}{\partial q^{i}} \dot{q}^{i}\dot{q}^{j},
    \end{equation*}
    which equals $2X^{c}(L_{g})$.
\end{proof}

Thus the contracting condition \eqref{contraction:condition} becomes
\begin{equation}\label{new:contraction:condtion}
	X^{c}(L_{g}) \leqslant -2\lambda L_{g}
\end{equation}

Therefore, the contracting condition can be written in terms of the contraction metric and its complete lift, and we can deduce when a system is contracting by only looking at the contraction metric and the geometry of the tangent bundle.
\section{Contracting forced Lagrangian systems}

In this section we will consider a mechanical system on $Q=\mathbb{R}^{n}$ given by a Lagrangian function of the type
\begin{equation}\label{mechanical:Lagrangian}
        L(q,\dot{q})=\frac{1}{2}\dot{q}^{T}M\dot{q} - \frac{1}{2} q^{T}Kq,
\end{equation}
and subject to the dissipative force
    \begin{equation}\label{dissipative:force}
    F(q,\dot{q})=-b\dot{q}^{i}, \quad b>0.
    \end{equation}
with $b>0$. We start with an example.

\begin{example}
    Let $Q=\mathbb{R}^{3}$ and, in addition, suppose that $K=k I_{3}$ with $k>0$, $M=\text{diag}(m1,m2,m3)$ and $m1<m2<m3$. Therefore the dynamics is given by the vector field
    $$\Gamma(q,\dot{q})=\left( \dot{q}, -M^{-1}(kq + b\dot{q}) \right).$$
    Consider the Riemannian metric $\mathbb{G}$ on $T\mathbb{R}^{3}$ determined by the block matrix
    $$\mathbb{G}=\begin{bmatrix}
    k I_{3} & b\varepsilon I_{3} \\
    b\varepsilon I_{3} & M
    \end{bmatrix}$$
    Then the quadratic form defined by $\mathcal{L}_{\Gamma}\mathbb{G}$ is negative-definite as long as $\varepsilon\in \left[0, \frac{2km_{1}}{b^{2} + 2m_{1}k}\right]$, under the hypothesis $m_{1}<m_{2}<m_{3}$. Thus, by Theorem \ref{main:theorem} the tuple $(T\mathbb{R}^{3},\Gamma,\mathbb{G},T\mathbb{R}^{3})$ is a contracting system.
\end{example}

As a consequence, we make the following definition.

\begin{definition}
    Given a mechanical Lagrangian $L:TQ\rightarrow \mathbb{R}$ of the type \eqref{mechanical:Lagrangian} and the dissipative force \eqref{dissipative:force}, the \textit{mechanical contraction metric associated with $L$ and $F$ and parameter $\varepsilon$} is the Riemannian metric given by
    \begin{equation}\label{mechanical:contraction metric}
    \mathbb{G}_{\varepsilon} = \begin{bmatrix}
        K & b\varepsilon I_{n} \\
        b\varepsilon I_{n} & M
    \end{bmatrix}
\end{equation}
\end{definition}

To finish this section, we prove a general result for unconstrained mechanical systems. Before we state the theorem, we provide a useful Lemma from Linear Algebra.

\begin{lemma}\label{shur:lemma}
    Consider a symmetric real block matrix of the form
    $$X=\begin{bmatrix}
        A & B \\
        B^{T} & C
    \end{bmatrix}$$
    such that $C$ is invertible. Then $X$ is negative definite if and only if $C$ and $A-BC^{-1}B^{T}$ are negative definite.
\end{lemma}

See \cite{Boyd} for a proof. Last lemma is sometimes known as Schur's lemma in reference to the matrix $S=A-BC^{-1}B^{T}$ which is called Schur's complement.

\begin{theorem}\label{mechanical:theorem}
    Suppose $L:TQ\rightarrow \mathbb{R}$ is a Lagrangian of the type \eqref{mechanical:Lagrangian} subjected to the dissipative force given by \eqref{dissipative:force} and such that $K$ and $M$ are symmetric positive definite real matrices. In addition, suppose that $K$ and $M^{-1}$ commute.
    
    Let $\bar{q}$ be an isolated critical point of $\text{grad } V$ which is also an equilibrium point of the equations of motion. Then there exists a neighbourhood $\mathcal{U}$ of $0_{\bar{q}}\in TQ$ and a parameter $\varepsilon$ such that $(TQ,\Gamma,\mathbb{G}_{\varepsilon},\mathcal{U})$ is a contracting system for the mechanical contraction metric $\mathbb{G}$. 
\end{theorem}

\begin{remark}
    The condition that $K$ and $M^{-1}$ commute is achieved for instance when $K$ and $M$ are diagonal matrices.
\end{remark}

\begin{proof}
    Consider the Lagrangian function $L_{\mathbb{G}_{\varepsilon}}:T(TQ)\rightarrow \mathbb{R}$ determined by the kinetic energy. If $(q,\dot{q},v,\dot{v})$ are the natural bundle coordinates on $T(TQ)$ then
    $$L_{\mathbb{G}_{\varepsilon}}=\frac{1}{2} \left( v^{T}Kv + 2b\varepsilon \dot{v}^{T}v + \dot{v}^{T}M\dot{v} \right).$$
    The complete lift on the mechanical vector field $\Gamma$ is
    $$\Gamma^{c}=\left( \dot{v}, M^{-1}(-Kv-b\dot{v}) \right).$$
    Therefore,
    \begin{equation}\label{mechanical:expression}
        \Gamma^{c}(L_{\mathbb{G}_{\varepsilon}}) = -b\varepsilon v^{T} K M^{-1} v - b^{2} \varepsilon \dot{v}M^{-1}v + (b\varepsilon -b)\dot{v}^{T}\dot{v}.
    \end{equation}
    Using Theorem \ref{main:theorem} and Lemma \ref{useful:lemma}, we must prove that the (block) matrix
    $$\begin{bmatrix}
        -b\varepsilon KM^{-1} & -\frac{b^{2}\varepsilon}{2}M^{-1} \\
        -\frac{b^{2}\varepsilon}{2}M^{-1} & (b\varepsilon-b)I_{n}
    \end{bmatrix},$$
    where $I_{n}$ is the identity matrix and $n$ is the dimension of the configuration space $Q$, is negative definite. Using Lemma \ref{shur:lemma}, it is enough to show that $C:=(b\varepsilon-b)I_{n}$ is invertible and negative definite and $$D_{\varepsilon}:=-b\varepsilon KM^{-1} - \left(\frac{b^{2}\varepsilon}{2}M^{-1}\right)\left(\frac{1}{b\varepsilon-b}I_{n}\right)\left(\frac{b^{2}\varepsilon}{2}M^{-1}\right)$$ is negative definite.
    It is easy to see that $C$ is negative definite and invertible whenever $\varepsilon<1$. Hence, let us suppose that $\varepsilon<1$. It is not difficult to show that $D_{\varepsilon}$ is negative definite if and only if
    $$\tilde{D}_{\varepsilon}=\varepsilon(\varepsilon-1)KM^{-1}+\frac{b^{2}\varepsilon^{2}}{4}M^{-2}$$ is negative definite. Supposing that $\varepsilon>0$, then the condition that $\tilde{D}_{\varepsilon}$ is negative definite is equivalent to
    $$F_{\varepsilon} = \varepsilon\left( KM^{-1} + \frac{b^{2}}{4}M^{-2} \right) - KM^{-1}$$ being negative definite.
    
    If $K$ and $M^{-1}$ commute then $KM^{-1}$ is also positive definite. Hence, for $\varepsilon=0$, $F_{0}=-KM^{-1}$ is negative definite and for $\varepsilon=1$, $F_{1}=\frac{b^{2}}{4}M^{-2}$ is positive definite. Then by continuity reasons, there exists $\varepsilon\in]0, 1[$ such that $F_{\varepsilon}$ is negative definite.    
\end{proof}

\begin{remark}If $M, K$ are in the conditions of Theorem \ref{mechanical:theorem}, from Proposition \ref{contracting:lyapunov:prop}, a Lyapunov function for dissipative mechanical systems is
$$V(q,\dot{q})=\|\Gamma(q,\dot{q})\|_{\mathbb{G}_{\varepsilon}}^{2}.$$

Thus, for Lagrangian functions of the type \eqref{mechanical:Lagrangian} and dissipative forces of the type \eqref{dissipative:force}, the Lyapunov function is
\begin{equation*}
    \begin{split}
        V(q,\dot{q}) = & q^{T}KM^{-1}Kq + 2b\dot{q}^{T}(M^{-1}K-\varepsilon M^{-1}K)q \\
        & + \dot{q}^{T}(K + b^{2}M^{-1} - 2b^{2}\varepsilon M^{-1})\dot{q}
    \end{split}
\end{equation*}


Following Proposition \ref{contracting:lyapunov:prop}, $V$ is a strict Lyapunov function for the equilibrium point $(0,0)$ and all trajectories starting in nearby points converge exponentially fast to the equilibrium point.

\end{remark}

\section{Extension to contact Lagrangian systems}

The previous discussion might be extended to the case of contact system with Lagrangian of the type
\begin{equation}\label{contact:Lagrangian}
        L(q,\dot{q}, z)=\frac{1}{2}\dot{q}^{T}M\dot{q} - \frac{1}{2} q^{T}Kq - \frac{b}{m}z,
\end{equation}
with $b,m>0$, whose corresponding vector field is of the form
\begin{equation*}
    \begin{split}
        \Gamma(q,\dot{q}, z)=&\left( \dot{q}, -M^{-1}(Kq +  \frac{b}{m}M\dot{q}), \frac{1}{2}\dot{q}^{T}M\dot{q} \right. \\
        & \left. - \frac{1}{2} q^{T}Kq - \frac{b}{m}z \right).
    \end{split}
\end{equation*}

In this case, we will consider the \textit{contact contraction metric} on $TQ \times \mathbb{R}$ given by
\begin{equation}\label{contact:contraction metric}
    \mathbb{G}_{\varepsilon} = \begin{bmatrix}
        K & b\varepsilon I_{n} & 0 \\
        b\varepsilon I_{n} & M & 0 \\
        0 & 0 & b
    \end{bmatrix}
\end{equation}

\begin{theorem}\label{contact:theorem}
    Suppose $L:TQ\times \mathbb{R}\rightarrow \mathbb{R}$ is a  contact Lagrangian of the type \eqref{contact:Lagrangian} and such that $M$ is the diagonal matrix with all entries equal to $m>0$ and $K$ is a symmetric positive definite real matrix that commutes with $M^{-1}$.
    
    There exists a neighbourhood $\mathcal{U}$ of $(0,0)\in TQ\times \mathbb{R}$ and a parameter $\varepsilon$ such that $(TQ,\Gamma,\mathbb{G}_{\varepsilon},\mathcal{U})$ is a contracting system for the contact  contraction metric $\mathbb{G}_{\varepsilon}$. 
\end{theorem}

\begin{proof}
    First notice that since $M$ is the diagonal matrix with all entries equal to $m>0$, the equations of motion for are just
    $$\ddot{q} = M^{-1} \left(- Kq - b\dot{q}\right), \quad \dot{z}=L.$$
    Following the same steps as in the proof of Theorem \ref{mechanical:theorem}, we write $\Gamma^{c}(L_{\mathbb{G}_\varepsilon})$ where $\Gamma$ is now the contact vector field and $\mathbb{G}_{\varepsilon}$ is the contact contraction metric.

    Then
    $$\Gamma^{c}(L_{\mathbb{G}_\varepsilon}) = r_{0} -bw v^{T}K q + bw \dot{v}^{T}M\dot{q} - b^{2}w^{2},$$
    where $r_{0}$ is the expression obtained in \eqref{mechanical:expression}. Therefore, we must prove that the matrix
    $$\begin{bmatrix}
        -b\varepsilon KM^{-1} & -\frac{b^{2}\varepsilon}{2}M^{-1} & -\frac{1}{2}bKq \\
        -\frac{b^{2}\varepsilon}{2}M^{-1} & (b\varepsilon-b)I_{n} & \frac{1}{2}bM\dot{q} \\
        -\frac{1}{2}bq^{T}K & \frac{1}{2}b\dot{q}^{T}M & -b^{2}
    \end{bmatrix}$$
    is negative definite. Using Lemma \ref{shur:lemma} and the fact that the upper left diagonal matrix is negative definite (as we have already seen in the proof of Theorem \ref{mechanical:theorem}) all we have to check is that the real number $x=-b^{2} - B^{T}A^{-1} B$ is negative, where $A$ is the upper left $2n\times 2n$ matrix and $B^{T}=[-\frac{1}{2}bq^{T}K, \frac{1}{2}b\dot{q}^{T}M]$. Notice that the second term is a quadratic form evaluated on the vector $B$. If $(q,\dot{q})$ are restricted to a sufficiently small subset $\mathcal{U}$ containing $0$, then the norm of $B^{T}A^{-1} B$ can be made smaller than $b^{2}$. Thus on $\mathcal{U}$ we have that $x<0$.
\end{proof}

\begin{example}
    Consider the mechanical system given by the contact Lagrangian
    $$L=\frac{1}{2}\dot{q}^{2}-\frac{1}{2}q^{2} - bz.$$

    The associated contact dynamical system is
    $$\ddot{q} =- q - b \dot{q}, \quad \dot{z}=L.$$

    The contact contraction metric is now
    $$\mathbb{G}_{\varepsilon} = \begin{bmatrix}
        1 & b\varepsilon & 0 \\
        b\varepsilon & 1 & 0 \\
        0 & 0 & b
    \end{bmatrix}.$$
    And so we have to examine the nature of the quadratic form
    $$\Gamma^{c}(L_{\mathbb{G}_\varepsilon}) = -b\varepsilon v^{2} - b^{2} \varepsilon \dot{v}v + (b\varepsilon -b)\dot{v}^{2} - bq w v + b\dot{q} w \dot{v} - b^{2}w^{2},$$
    associated to the $3\times 3$ matrix
    $$\begin{bmatrix}
        -b\varepsilon & -\frac{b^{2}\varepsilon}{2} & -\frac{1}{2}bq \\
        -\frac{b^{2}\varepsilon}{2} & (b\varepsilon-b) & \frac{1}{2}b\dot{q} \\
        -\frac{1}{2}bq^{T} & \frac{1}{2}b\dot{q}^{T} & -b^{2}
    \end{bmatrix}.$$
    If $0<\varepsilon<\frac{2}{b^{2} + 2}$, then the upper left $2\times 2$ matrix is negative definite. To show that the matrix above is also negative definite we need in addition that
    $$x=-b^{2} - B^{T}A^{-1} B<0,$$
    where $B$ is the vector with coordinates $(-\frac{1}{2}bq, \frac{1}{2}b\dot{q})$. Since the matrix $A^{-1}$ also defines a quadratic form, if we squeeze the vector $(q,\dot{q})$ to be small enough, then $B$ will be also close to zero and $x$ will approach the value $-b^{2}$. Therefore, there is a neighbourhood of $0$ in which the matrix above is negative definite.
\end{example}

\begin{remark}As a sub-product of our approach, we may study the stability of contact systems with Lagrangian functions of the type \eqref{contact:Lagrangian} and where $M, K$ are in the conditions of Theorem \ref{contact:theorem}. In particular, from Proposition \ref{contracting:lyapunov:prop}, a Lyapunov function for these systems is
$$V(q,\dot{q},z)=\|\Gamma(q,\dot{q},z)\|_{\mathbb{G}_{\varepsilon}}^{2}.$$

In the last example, the Lyapunov function is
$$V(q,\dot{q},z)=(1-2b^{2}\varepsilon)\dot{q}^{2}-2b\varepsilon\dot{q}q + b\left( \frac{1}{2}\dot{q}^{2}-\frac{1}{2}q^{2} - bz \right)^{2}.$$

Following Proposition \ref{contracting:lyapunov:prop}, $V$ is a strict Lyapunov function for the equilibrium point $(0,0,0)$ and all trajectories starting in nearby points converge exponentially fast to the equilibrium point.

\end{remark}

\section{Application to reduced nonholonomic systems with symmetry}

It is a well-known fact that when we have a symmetry, the number of equations describing the motion of nonholonomic systems may be reduced. Under certain conditions, the reduced system is Hamiltonian, meaning that it may be seen as the Hamiltonian vector field with respect to some Hamiltonian function defined on the reduced space (\cite{CaLeMaMa, BaSn, SCSarlet99}).

Consider, as an example, the vertical rolling disk with potential given by the Lagrangian function
\begin{equation*}
    L=\frac{m}{2}\left( \dot{x}^2 + \dot{y}^2\right) + \frac{I}{2}\dot{\theta}^2 + \frac{J}{2}\dot{\varphi}^2 - \frac{1}{2}(\theta + \varphi)
\end{equation*}
together with the non-slipping constraints $\dot{x}=R\dot{\theta}\cos \varphi$, $\dot{y}=R\dot{\theta}\sin \varphi$ generating the distribution
$$\mathcal{D}= \left\langle \left\{ \frac{\partial}{\partial \theta} + R\cos \varphi \frac{\partial}{\partial x} + R \sin \varphi \frac{\partial}{\partial y}, \frac{\partial}{\partial \varphi} \right\} \right\rangle$$
This system is invariant under the action of translations, meaning that the group $\mathbb{R}^{2}$ acts on the configuration manifold $Q=\mathbb{R}^{2}\times \mathbb{S}^{1}\times \mathbb{S}^{1}$, according to the map $\Phi_{(r,s)}:(x,y,\theta,\varphi) \mapsto (x+r,y+s,\theta,\varphi)$. Then both the Lagrangian function as well as the distribution are invariant with respect to the tangent lift of this action $T\Phi_{(r,s)}:TQ \rightarrow TQ $. In fact, one may prove that this is an example of a Chaplygin system (\cite{CaLeMaMa}) associated to a projection $\pi:Q\rightarrow \overline{Q}:= Q/\mathbb{R}^{2}\equiv \mathbb{S}^{1}\times \mathbb{S}^{1}$.


The complete equations of motion might be reduced using the projection by $\pi$. The reduced equations for this nonholonomic system are
\begin{equation*}
    \begin{cases}
        (I+m R)\ddot{\theta} = - \theta \\
        J \ddot{\varphi} = - \varphi
    \end{cases}
\end{equation*}
which is a Lagrangian system with respect to the reduced Lagrangian $l:T\overline{Q}\rightarrow \mathbb{R}$
\begin{equation*}
    l = \frac{(I+m R)}{2}\dot{\theta}^2 + \frac{J}{2}\dot{\varphi} - \frac{1}{2}(\theta + \varphi).
\end{equation*}
In addition, suppose that a cyclic dissipative force is acting on the nonholonomic system in $TQ$ of the type
$$F= -b\left( \dot{\theta} d\theta + \dot{\varphi} d\varphi\right).$$
This external force respects the symmetry of the system and appears in the reduced equations of motion giving
\begin{equation}\label{reduced:equations:disk}
    \begin{cases}
        (I+m R)\ddot{\theta} = - \theta - b \dot{\theta} \\
        J \ddot{\varphi} = - \varphi - b \dot{\varphi}.
    \end{cases}
\end{equation}
The reduced system \eqref{reduced:equations:disk} is a mechanical system with a dissipative force in the conditions of Theorem \ref{mechanical:theorem}. Therefore there is a neighbourhood $\mathcal{U}\subseteq T\overline{Q}$ of the equilibrium point $(0,0,0,0)$ where the system is contracting for the metric
$$\mathbb{G}_{\varepsilon} = \begin{bmatrix}
        1 & 0 & b\varepsilon & 0 \\
        0 & 1 & 0 & b\varepsilon \\
        b\varepsilon & 0 & I+m R & 0 \\
        0 & b\varepsilon & 0 & J
    \end{bmatrix}$$
on $T\overline{Q}$. Therefore, using Proposition \ref{contracting:lyapunov:prop}, we conclude that trajectories of the reduced dynamics starting in $\mathcal{U}$ converge exponentially fast to the equilibrium point, with respect to the mechanical contraction metric $\mathbb{G}_{\varepsilon}$ on $T\overline{Q}$ and for a suitable value of $\varepsilon$.

On the original dynamics on $TQ$, this implies that the submanifold $\mathcal{EP}\subseteq \mathcal{D}$, locally defined by $\theta=0, \varphi=0, \dot{\theta}=0, \dot{\varphi}=0$ and the corresponding nonholonomic constraints $\dot{x}=0$, $\dot{y}=0$, is composed by equilibrium points of the nonholonomic system.



Now, consider in the original tangent bundle $TQ$ a Riemannian metric $\mathcal{G}$ satisfying the following conditions:
\begin{equation}\label{lifted:metric}
    \begin{split}
        & \mathcal{G}(X^{H},Y^{H}) = \mathbb{G}(X,Y), \\
        & \mathcal{G}(X^{H},Z) = 0, \quad \forall X, Y \in \mathfrak{X}(T\bar{Q}) \text{ and } Z \in V,
    \end{split}
\end{equation}
where $V$ is the vertical distribution. There exists more than one Riemannian metric satisfying these conditions. Note that the value of $\mathcal{G}$ on the vertical bundle is not fixed but it must be a symmetric non-degenerate tensor on vectors belonging to the vertical distribution. To such a choice of Riemannian metric on $TQ$ we call an horizontal lift of the contraction metric $\mathbb{G}_{\varepsilon}$.

Using any horizontal lift of the metric $\mathbb{G}_{\varepsilon}$, we may deduce that any solution starting in a tubular neighborhood around the equilibrium submanifold $\mathcal{EP}$ converges exponentially fast to $\mathcal{EP}$. Here, the distance of a point $v\in \mathcal{D}$ to $\mathcal{EP}$ is given by
$$d_{\mathcal{G}}(v,\mathcal{EP}) = \inf_{w \in \mathcal{EP}} d_{\mathcal{G}}(v,w).$$

\begin{theorem}\label{nh:constraction:theorem}
    Suppose that $(L,\mathcal{D})$ is a Chaplygin nonholonomic system and suppose that the reduced dynamics is a mechanical system associated to a Lagrangian of the type \eqref{mechanical:Lagrangian}. In addition, suppose that $F$ is a cyclic force map and in the reduced equations appears in the form \eqref{dissipative:force}.
    
    Then there exists a tubular neighbourhood $\mathcal{U}$ of the submanifold $\mathcal{EP}$ of equilibrium points of the forced nonholonomic system $(L,\mathcal{D}, F)$ such that any nonholonomic solution starting in $\mathcal{U}$ converges exponentially fast to $\mathcal{EP}$, with respect to a metric $\mathcal{G}$ satisfying \eqref{lifted:metric}.
\end{theorem}

\begin{proof}
    The distance of a point $v$ in $\mathcal{D}$ to the equilibrium manifold as measured by $\mathcal{G}$ is always smaller than the length of an integral curve of $\Gamma_{nh}$ that joins $v$ and a point in the equilibrium submanifold. This curve exists in a tubular neighborhood of the submanifold $\mathcal{EP}$, since it is the horizontal lift of a solution of the reduced dynamics that finishes in the equilibrium point. In addition, since this curve is horizontal, its length is the same as the length of the reduced curve as measured by the mechanical contraction metric $\mathbb{G}_{\varepsilon}$, which tends to zero exponentially fast. Therefore, the distance on $TQ$ must also converge exponentially fast.
\end{proof}

From Theorem \ref{nh:constraction:theorem}, we deduce that any trajectory of the vertical rolling disk starting sufficiently close to the equilibrium submanifold
$$\mathcal{EP} = \{ (x,y,\theta,\varphi,\dot{x},\dot{y},\dot{\theta},\dot{\varphi})\in \mathcal{D} | \theta=0, \varphi=0, \dot{\theta}=0, \dot{\varphi}=0\}$$
converges exponentially fast to a point in $\mathcal{EP}$, with respect to an horizontal lift of the contracting metric.

\section{Conclusions}

We have addressed the problem of identifying contracting systems among dynamical systems
appearing in mechanics. We provide a sufficient condition to identify contracting systems in a general Riemannian manifold and we applied this technique to establish that particular types of dissipative systems are contracting as well as we have stated immediate consequences of this fact for the stability of these systems. For future work, the results of this paper could be extension to general potentials and several types of mechanical systems with symmetries (not necessarily constrained). One of the most interesting applications of this approach could be the study of the stability of multi-agent formation problems through contraction theory. 

In a different direction, it would be also interesting to show how to adapt our results to dissipative systems in general Riemannian manifolds and especially to Lie groups, since many problems in robotics evolve on this type of manifolds.

\nocite{*}

\bibliography{newbib}
\bibliographystyle{IEEEtran}

\appendix

\begin{lemma}\label{norm:proof}
Let $M$ be a manifold and $\overline{\mathcal{U}}\subseteq M$ a compact subset. Then
    $$\|h\|=\sup_{x\in \tilde{\mathcal{U}}}\sup_{\|v_{x}\|_{g}=1} |h(v_{x},v_{x})|$$
is a norm on the space of symmetric $(0,2)$-tensors on $\overline{\mathcal{U}}$.
\end{lemma}

\begin{proof}
    We will prove that $\|h\|$ is in fact a norm on the space of symmetric $(0,2)$-tensors on $\overline{\mathcal{U}}$. First of all, this norm has a well-defined value for each tensor since the supreme are measured over compact sets. The-non trivial conditions to be satisfied are the triangle inequality and positive definiteness, i.e., $\|h\|=0\implies h=0.$
    
    Triangle inequality follows from the triangle inequality for the absolute value $|\cdot|$. If $\|h\|=0$, we have that $h(v_{x},v_{x})=0$ for all $v_{x}\in T_{x}M$ and all $x\in \overline{\mathcal{U}}$. Using this fact together with symmetry of $h$ we have
\begin{equation*}
    \begin{split}
        h(v_{x},w_{x}) & = h(v_{x},v_{x}) + 2h(v_{x},w_{x}) + h(w_{x},w_{x}) \\
        & =h(v_{x}+w_{x},v_{x}+w_{x})=0.
    \end{split}
\end{equation*}
Therefore $h=0$.
\end{proof}

\clearpage

\end{document}